\documentclass[12pt,A4,reqno]{amsart}
\usepackage{amsfonts}
\usepackage{mathrsfs}
\usepackage[T1]{fontenc}
\usepackage{amssymb}
\usepackage{amsmath,amscd}
\usepackage{color}
\usepackage{epsf}
\usepackage{graphicx}

\theoremstyle{plain}
\newtheorem{thm}{Theorem}[section]
\newtheorem{lem}[thm]{Lemma}
\newtheorem{prop}[thm]{Proposition}

\theoremstyle{definition}

\newtheorem{exam}{Example}

\newcommand{\R}{\mathbb R}
\newcommand{\Z}{\mathbb Z}

\begin{document}

\title [\ ] {on cobordism of generalized (real) bott manifolds}

\author{Yuxiu Lu}
\address{Department of Mathematics and IMS, Nanjing University, Nanjing, 210093, P.R.China
 }
 \email{luyxnj@gmail.com}
\date{September 17, 2017}


\begin{abstract}
  We show that all generalized (real) Bott manifolds which are (small covers) quasitoric manifolds over a product of simplices $\Delta^{n_1}\times\cdots\times\Delta^{n_r}\times\Delta^{1}$ are always boundaries of some manifolds. But these manifolds with the natural $(\mathbb{Z}_2)^n$ action do not necessarily bound equvariantly. In addition, we can construct some examples of null-cobordant but not orientedly null-cobordant manifolds among quasitoric manifolds.
 \end{abstract}

\maketitle

 \section{Introduction}

In this paper, we will mainly investigate cobordism of certain type of manifolds called real Bott manifold and generalized real Bott manifolds (see Kamishima and Masuda~\cite{Kamishima09}).
A real Bott manifold is defined to be a sequence of $\R P^1$-bundles
   \begin{align*}
         M_n\overset{\R P^1}\longrightarrow M_{n-1}\overset{\R P^1}{\longrightarrow}\cdots  \overset{\R P^1} \longrightarrow M_1\overset{\R P^1}\longrightarrow M_0=\{a\ point\}
       \end{align*} 
such that  $M_i\longrightarrow M_{i-1}$ for $i=1,\cdots,n$ is the projective bundle of a Whitney sum of two real line bundles over $M_i$.
It is well known that a real Bott manifold is a flat manifold, which means that it admits a Riemannian metric with everywhere zero sectional curvature. Then this manifold is a boundary, since any closed manifold with a flat  Riemannian metric is a boundary (see G.C.~Hamrick and D.C.~Royster, ~\cite{Hamrick82}). 

Any real Bott manifold bounds equivariantly  (see Y.~Cheng and Y.~Wang~\cite {Wang11}, Z.~L\"u and  Q.~Tan~\cite {LuZhi14}). But their methods are not easy to use to discuss a generalized real Bott manifold. So we need some new technique to deal with the general case.

     This study of $n$- dimensional real Bott manifolds is equivalent to the study of small covers over an $n$-cube (see Kamishima~\cite {Kamishima09}, Suh~\cite {Suh09}, Masuda~\cite {Masuda91}). This kind of manifolds and quasitoric manifolds are defined in Davis-Januszkiewicz~\cite {Davis91}. By a small cover one means an $n$-dimensional manifold $M^n$ with $\Z_2^n$-action locally isomorphic to the standard action of $\Z_2^n$ on $\R^n$ for which the orbit space $M^n/\Z_2^n$ is  a simple convex polytope $P^n$. A quasitoric manifold means a smooth orientable $2n$-dimensional manifold $M^{2n}$ with $T^n$-action locally isomorphic to the standard action of $T^n$ on $\mathbb{C}^n$ for which the orbit space $M^{2n}/T^n$ is a simple convex polytope $P^n$. The facial submanifold of a quasitoric manifold or a small cover is defined to be $\pi^{-1}(F)$, where $\pi:M^{dn}\to P^n$ is a quotient map, $d=1,2$. The concept of characteristic map is very important in the construction of the small cover and quasitoric manifold. Take a small cover as an example. 
A characteristic map means a function $\lambda: \{the\ facets \ of \ P^n\} \to \Z_2^n$, which satisfies that for every vertex $v=F_1\cap \cdots \cap F_n$ of $P^n$, $\lambda(F_1),\cdots,\lambda(F_n)$ form a basis of $\Z_2^n$. The characteristic map can also be identified with a map mapping each facet $F$ to an element in $\Z_2^n$, where $\lambda(F)$ is the generator of the isotropy group of the facial submanifold $\pi^{-1}(F)$ for each facet $F$. The characteristic map determines a small cover over $P^n$ and that any function satisfying the above condition can be realized as the characteristic function of some small cover.

In their paper, Davis and Januszkiewicz also provide following formulas to compute the total Pontryagin class and the total Stiefel-Whitney class for small covers and quasitoric manifolds. 
 \begin{thm}
  (Davis and Januszkiewicz \cite [corollary 6.8]{Davis91}) Let $M^n$ be a small cover over a simple polytope $P^n$. Let $v_1,\cdots, v_m$ denote elements in the cohomology ring of $M$ dual to the facial submanifolds of all the facets of $M$. Let $j: M^n\to {B_d}P^n$ be inclusion of the fiber.\\ 
(i) If $d=1$, then
     \begin{align*}
      w(M^{dn})={j^*}\prod _{i=1}^{m}(1+v_i),   \  and \   \      
      p(M^{dn})=1.
       \end{align*}  
(ii)If $d=2$, then 
     \begin{align*}
      w(M^{dn})={j^*}\prod _{i=1}^{m}(1+v_i) \ mod\ 2,  \ and\   \      
      p(M^{dn})={j^*}\prod _{i=1}^{m}(1-v_i^2) .
       \end{align*}             
  \end{thm}
Davis and Januszkiewicz also give a method to compute the cohomology ring of small covers and quasitoric manifolds. 
 \begin{thm}
(Davis and Januszkiewicz \cite [Theorem 4.14]{Davis91}) \ If $M^n$ is a small cover, then $H^*(M;\Z_2)=\Z_2[v_1,\cdots,v_m]/(I+J)$; If $M^{2n}$ is a quasitoric manifold, then $H^*(M;\Z)=\Z[v_1,\cdots,v_m]/(I+J)$; $I$ is the Stanley-Reinser ideal of $P^n$, $J$ is generated by linear combinations of $v_1,\cdots,v_m$ which are determined by characteristic map $\lambda$.
  \end{thm}
By these two theorems, we can prove that any real Bott manifold is cobordant to zero, and more generally, any generalized real Bott manifold over $P^n\times \Delta^1$ is cobordant to zero, where $P^n$ is an arbitrary product of simplices (see Theorem 3.4). This gives some examples of small covers which are boundaries while not bounding equivariantly. Furthermore, one sufficient condition for a generalized real Bott manifold over $P^n\times \Delta^k$ to be null-cobordant, is given for odd numbers $k=2^l-1$ (see Theorem 3.6). Finally, for any product of simplices $P^n$, the generalized Bott manifold over $P^n\times \Delta^1$ is proved to be orientedly cobordant to zero just as the case of generalized real Bott manifold (see Theorem 4.3), and we can confirm that a special kind of quasitoric manifolds over an even dimensional cube are all unorientedly cobordant to zero, but are not orientedly cobordant to zero (see Theorem 4.5).

  \section{Cobordism of real Bott manifold}
     Suppose that $M^n$ is a small cover over an $n$-cube.
     Let the first $\Z_2$-coefficient cohomology classes $v_1,\cdots, v_n$ dual to 
     the facial submanifolds of facets $F_1,\cdots,F_n$, meeting at a vertex.
Then we have facets $F_1^*,\cdots,F_n^*$ which are parallel to $F_1,\cdots,F_n$ respectively. Let the first $\Z_2$-coefficient cohomology classes $u_1,\cdots,u_n$ be dual to the facial submanifolds of facets $F_1^*,\cdots,F_n^*$. Since 
$\{\lambda(F_1),\cdots,\lambda(F_n)\}$ is a basis of $\Z_2^n$, we can assume that 
\  $\lambda(F_1)=e_1,\cdots, \lambda(F_n)=e_n$. So $\lambda(F_1^*),\cdots,\\ \lambda(F_n^*)$ can be written in the form of linear combinations of $e_1,\cdots,e_n$, which will give a $\Z_2$-coefficient matrix of order $n$. This matrix is defined to be the \emph{reduced matrix} of the small cover $M^n$, denoted by $A=(a_{ij})$, where 
 \begin{align}
\lambda(F_i^*)=\sum{a_{ij}}e_j. 
 \end{align}
The theorem of Davis and Januszkiewicz gives the total Stiefel-Whitney class $w= (1+u_1)(1+v_1)\cdots (1+u_n)(1+v_n)$. By the restrictions given by the Stanley-Reinser ideal $I$, 
$w= (1+u_1+v_1)\cdots (1+u_n+v_n)$. \\
Define $y_i=u_i+v_i$, then 
$w=(1+y_1)\cdots (1+y_n)$.\\ 
The restrictions given by the ideal $J$ determines another relation:
 \begin{align}
(u_1,\cdots,u_n)(A+E_n)= (y_1,\cdots,y_n),  
 \end{align}
where $E_n$ is the identity matrix of order $n$.
\begin{lem}
Let $A$ be an $n \times n$ matrix with entries in ${\Z_2}$. Suppose that every principal minor of A is $1$. If $det A=1$, then A is conjugated by a permutation matrix to a unipotent upper triangular matrix.
  \end{lem}
\begin{proof} see M.~Masuda and T.E.~Panov\cite[Lemma 3.3] {Masuda91}.
\end{proof}
  \begin{lem}
 The reduced matrix of a small cover over $n$-cube can be conjugated by a permutation matrix to a unipotent upper triangular matrix.
  \end{lem}
\begin{proof} Owing to the linear independence of characteristic maps of facets in every vertex, we can easily get that every principal minor of A is $1$. So, by the Lemma 2.1 above, the conclusion is obvious.
\end{proof}

\begin{lem}
     If we assume the conditions above, and if $l_1+\cdots+l_k=k+1$, where $1\le k\le n$.
     Then for 1-cohomology classes $u_1,\cdots,u_n$ of $M^n$, $u_1^{l_1}\cdots u_k^{l_k}=0$.
 \end{lem}
\begin{proof} If $k=n$, the equation is obvious because the degree of $u_1^{l_1}\cdots u_k^{l_k}$ is greater than the dimension of small cover.
  If $k< n$, this theorem can be proved by induction. When $k=1$, then $l_1=2$, the only equation one should 
  justify is $u_1^2=0$, which will be proved in Lemma 3.3 in a broader sense. By the Lemma 2.2, we can assume that the \emph{reduced matrix} of a small cover is a unipotent upper-diagonal matrix. This means that $y_k$ can be expressed by a linear combination of $u_1,\cdots,u_{k-1}$ , $1\le k\le n$. Assume that all the equations of number $k$ is correct, then consider equations of number $k+1$.
     \begin{align*}
      u_1^{l_1}\cdots u_{k+1}^{l_{k+1}}=u_1^{l_1}\cdots u_k^{l_k}u_{k+1}y_{k+1}^{l_{k+1}-1}.
       \end{align*}
  The degree of $u_1^{l_1}\cdots u_k^{l_k}y_{k+1}^{l_{k+1}-1}$ is $l_1+\cdots+l_{k+1}-1=k+1$.
  Since $y_{k+1}$ can be expressed by $u_1,\cdots,u_k$ ,  $u_1^{l_1}\cdots u_k^{l_k}y_{k+1}^{l_{k+1}-1}$ consists of  monomials with elements in ${u_1,\cdots,u_k}$ of degree $k+1$. So by induction, these 
polynomials are all zero, and $u_1^{l_1}\cdots u_{k+1}^{l_{k+1}}$ is zero.
\end{proof}
 \begin{lem}
     Given the conditions above, if $l_2+\cdots+l_k=k$,  then $y_2^{l_2}\cdots y_k^{l_k}=0$.
 \end{lem}
\begin{proof} $y_2^{l_2}\cdots y_k^{l_k} $ consists of monomials with elements in ${u_1,\cdots,u_{k-1}}$ of degree $k$. So due to the above Lemma 2.3, we can obtain this equation.
\end{proof}

 \begin{thm}
      All the Stiefel-Whitney numbers of a small cover $M$ over n-cube is $0$, i.e. such a small cover is a boundary.
  \end{thm}
\begin{proof} Since $y_1=0$, the total Stiefel-Whitney class of $M$, $w=(1+{y_2})\cdots(1+{y_n})$. It is obvious that when $i_1+2i_2+\cdots+ni_n=n$,  $w_1^{i_1} \cdots w_n^{i_n}$consists of ploynomials in $y_2,\cdots, y_n$ of degree $n$. By the Lemma 2.4 above and the definition of Stiefel-Whitney numbers, this conclusion is proved.
\end{proof}

    \section{Cobordism of generalized real Bott manifold}
In this section, we will mainly use the terms defined in S.~Choi, M.~Masuda and D.Y.~Suh~\cite {Suh09}. Firstly, generalized real Bott manifold $M^n$ is defined to be a sequence of $\R P^k$-bundles of various numbers,
   \begin{align*}
         M_m\overset{\R P^{n_m}}\longrightarrow M_{m-1}\overset{\R P^{n_{m-1}}}{\longrightarrow}\cdots  \overset{\R P^{n_2}} \longrightarrow M_1\overset{\R P^{n_1}}\longrightarrow M_0=\{a\ point\},
       \end{align*}
such that  $M_i\longrightarrow M_{i-1}$ for $i=1,\cdots,m$ is the projective bundle of a Whitney sum of $n_i$ real line bundles over $M_i$. 
This definition is equivalent to another one that a generalized real Bott manifold is a small cover over a product of simplices $P^n$, where  $ P^n=\prod_{i=1}^{m}  \Delta^{n_i}$, with $\sum ^{m}_{i=1}{n_i}=n$, and $\Delta^{n_i}$ is the ${n_i}$-simplex for $i=1\cdots m$. Then each facet of $P^n$ is the product of a codimension-one face of one of the $\Delta^{n_i}$'s and the remaining simplices. Thus there are $n+m$ facets of $P$, which form the set
$\{ F_{k_i}^i| 0\le k_i\le n_i, 1\le i \le m \}$, where $F_{k_i}^i=\Delta^{n_1}\times\cdots\times\Delta^{n_{i-1}}\times f_{k_i}^i\times\Delta^{n_{i+1}}\times\Delta^{n_m}$. 
At each vertex $v_{{j_1}\cdots {j_m}}$, all facets except $\{ F_{j_i}^i| 1\le i\le m\}$ intersect with this vertex, where $v_{ {j_1}\cdots{j_m} }=v_{j_1}^1\times \cdots \times v_{j_m}^m$ and 
$v_{j_k}^k$ is the only vertex of $\Delta^{n_k}$ not contained in $f_{j_k}^k$.\\
Let $\lambda: \{ F_{k_i}^i| 0\le k_i\le n_i, 1\le i \le m \}\to \Z_2^n$ be the characteristic map of the small cover $M^n$,  the linear independence condition at vertex $v_{0\cdots0}$ means that 
$\lambda(F_1^1),\cdots,\lambda(F_1^{n_1}),\cdots,\lambda(F_m^1),\cdots,
\lambda(F_m^{n_m})$ consist of a basis of $\Z_2^n$, which is defined to be $\{e_1,\cdots,e_n\}$. Then $F_i^0$ for every $1\le i\le m$ can be written as a linear combination of 
$e_1,\cdots,e_n$.
We set $\lambda(F_i^0)=\mathbf{a}_i $ for $1\le i\le m$, where $\mathbf{a}_i $ is the coordinate of $\lambda(F_i^0)$ related to the basis $\{e_1,\cdots,e_n\}$.
In this way, we have the correspondence $m\times m$ \emph{reduced vector matrix} $A$ of small cover  
 $M^n$.
     \begin{align*}
     A=\begin{pmatrix}
        \mathbf{a}_1\\
        \vdots  \\
       \mathbf{a}_m \\  \end{pmatrix}
 =\begin{pmatrix}
        \mathbf{a}_1^1 & \cdots & \mathbf{a}_1^m   \\
        \vdots & \cdots & \vdots  \\
       \mathbf{a}_m^1& \cdots & \mathbf{a}_m^m\\  \end{pmatrix},   
where \ \mathbf{a}_i^j=(a_{i 1}^j,\cdots,a_{i n_i}^j).
 \end{align*}

We have more general versions of Lemma 2.1 and Theorem 2.2, which can be found in S.~Choi, M.~Masuda and D.Y.~Suh~\cite {Suh09} and also M.~Masuda and T.E.~Panov~\cite {Masuda91}. To state this lemma, the principal minor of a vector matrix is defined to be a principal minor of $A_{{k_1}\cdots{k_m}}$, where $1\le k_1\le n_1,\cdots,1\le k_m\le n_m$. $A_{{k_1}\cdots{k_m}}$ is defined in S.~Choi, M.~Masuda and D.Y.~Suh~\cite {Suh09}. For given $1\le k_j\le n_j$ with $j=1,\cdots,m$, $A_{{k_1}\cdots{k_m}}$ is the $m\times m$ submatrix of $A$ whose $j$-th column is the $k_j$-th column of the $m\times n_j$ matrix $({\mathbf{a}_1^j}^T,\cdots,{\mathbf{a}_m^j}^T)^T$. Thus 
     \begin{align*}
A_{{k_1}\cdots{k_m}}= \begin{pmatrix}
        a_{1{k_1}}^1& \cdots &a_{1{k_m}}^m   \\
              \vdots &   & \vdots \\
       a_{m{k_1}}^1 & \cdots & a_{m{k_m}}^m\\
   \end{pmatrix}. 
     \end{align*}

  \begin{lem}(S.~Choi, M.~Masuda and D.Y.~Suh~\cite {Suh09})
Let $A$ be an $m\times m$ reduced vector matrix with $\Z_2$-coefficient such that all the principal minors of $A$ are $1$, then $A$ is conjugate to a unipotent upper triangular vector matrix of the following form:
     \begin{align*}
 \begin{pmatrix}
        \mathbf{1} & \mathbf{b_1^2}&\mathbf{b_1^3} &\cdots & \mathbf{b_1^m}   \\
       \mathbf{0} & \mathbf{1}& \mathbf{b_2^3} &\cdots & \mathbf{b_2^m}  \\
        \vdots & \cdots &\cdots & \cdots  & \vdots \\
      \mathbf{0} & \cdots & \cdots & \mathbf{1} &\mathbf{b_{m-1}^m}\\
        \mathbf{0} & \cdots & \cdots & \mathbf{0}  & \mathbf{1}\\
   \end{pmatrix}  
     \end{align*}
where $\mathbf{0}=(0,\cdots,0)$, $\mathbf{1}=(1,\cdots,1)$ of appropriate sizes.
  \end{lem}

  \begin{lem}
 The reduced vector matrix of a small cover $M^n$ over a product of simplices $P^n$ can be conjugated by a permutation vector matrix to a unipotent upper triangular vector matrix.
  \end{lem}
\begin{proof} By the linear independence of characteristic maps in vertex $v_{{k_1}\cdots{k_m}}$, the determinant of $A_{{k_1}\cdots{k_m}}$ is 
nonzero. More generally,  because of the linear independence of characteristic maps in vertex $v_{{k_1}\cdots{k_l}0\cdots0}$ for $l\le m$, the principal minors of $A_{{k_1}\cdots{k_m}}$ are nonzero, which leads to the conclusion.
\end{proof}

 Let the $1$-cohomology classes $v_i^{(0)},\cdots,v_i^{(n_i)}$ dual to 
     the facial submanifolds of facets $F_{0}^i,\cdots,F_{n_i}^i$ for $1\le i \le m$. For convenience, we can define: $u_i=v_i^{(0)}$. From the restrictions given by $J$ and the unipotent upper triangular vector matrix, $v_i^{(j)}$ is the sum of $u_i$ and the linear combination of $u_1,\cdots,u_{i-1}$. In the following theorems, we will use $y_j$ defined in $v_i^{(j)}=u_i+y_j$ if there is no ambiguity. Then $y_j$ is a linear combination of $u_1,\cdots, u_{j-1}$.

Given these conditions, we can extend the Lemma 2.3 of the last section to a broader one. 
  \begin{lem}
   For every positive integer $k\le m$, if  $l_1+\cdots+l_k=\sum ^{k}_{i=1}{n_i}+1$, 
     \begin{align}
 u_1^{l_1}\cdots u_k^{l_k}=0.
     \end{align}
  \end{lem}
\begin{proof}  It is obvious to prove the equation when $k=m$. While $k<m$, it can be proved by induction. When $k=1$, the equation is equivalent to $u_1^{n_1+1}=0$. Because of the elements from the first column to the $n_1$-th column of the \emph{reduced matrix}, we have: $v_1^{(j)}=v_1^{(0)}=u_1$, the restriction of the Stanley-Reinser ideal means $0=\prod _{j=0}^{n_1}v_1^{(j)}=u_1^{n_1+1}$. Assume that the equation is true at every number $\le k$, then consider $u_1^{l_1}\cdots u_{k+1}^{l_{k+1}}$. If $l_{k+1}\le n_{k+1}$, then $l_1+\cdots+l_k \ge n_1+\cdots+n_k+1$, it can be proved by the induction hypothesis. If $l_{k+1}\ge n_{k+1}+1$, we first have: 
     \begin{align*}
0=v_{k+1}^{(0)}\cdots v_{k+1}^{(n_{k+1})}=u_{k+1}(u_{k+1}+y_1)\cdots (u_{k+1}+y_{n_{k+1}})\\
=u_{k+1}^{n_{k+1}+1}+u_{k+1}^{n_{k+1}}a_1+\cdots+u_{k+1}a_{n_{k+1}},
 \end{align*}
where $a_1,\cdots ,a_{n_{k+1}}$ are some homogenous polynomials in elements $u_1,\cdots ,u_k$, $deg(a_i)=i$, since $y_j$ is a linear combination of $u_1,\cdots, u_{j-1}$.
     \begin{align*}
       u_{k+1}^{l_{k+1}}=-u_{k+1}^{l_{k+1}-(n_{k+1}+1)}(u_{k+1}^{n_{k+1}}a_1+\cdots+u_{k+1}a_{n_{k+1}}),
 \end{align*}
Iterate this process, $u_{k+1}^{l_{k+1}}$ can be  written as 
 $u_{k+1}^j f_j$, where $deg(f_j)=l_{k+1}-j$ and $j\le n_{k+1}$. Since $f_j$ is a homogenous polynomial over $u_1,\dots, u_k$, we can compute the degree of $f_j u_1^{l_1}\cdots u_k^{l_k}$:
      \begin{align*}
                                deg(f_j{u_1^{l_1}}\cdots u_k^{l_k})=l_1+\cdots+l_{k+1}-j\ge n_1+\cdots+n_k+1,
 \end{align*}
which means $f_j{u_1^{l_1}}\cdots u_k^{l_k}$ is zero by induction. 
\end{proof}

Then we can use this lemma to give a more general theorem.
 \begin{thm}
      Let $P^n$ be any product of simplices. Then any small cover $N$ over $P^n\times \Delta^1$ has identically zero Stiefel-Whitney numbers, i.e. is a boundary.
  \end{thm}
\begin{proof}  The Stiefel-Whitney class of $N, \\
w(N)=(\prod_{i=1}^{m}(1+v_i^{(0)})\cdots(1+v_i^{(n_i)}))(1+v_{m+1}^{(0)})(1+v_{m+1}^{(1)})$.  
      \begin{align*}
      (1+v_{m+1}^{(0)})(1+v_{m+1}^{(1)})=1+v_{m+1}^{(0)}+v_{m+1}^{(1)},       
      \end{align*}
where $v_{m+1}^{(0)}+v_{m+1}^{(1)}$ is a linear combination of $u_1\cdots u_m$. 
So when $i_1+2i_2+\cdots+(n+1)i_{n+1}=n+1$,  
         $w_1^{i_1} \cdots w_{n+1}^{i_{n+1}}$ consists of forms like $u_1^{l_1}\cdots u_m^{l_m}$, where $l_1+\cdots+l_m=\sum ^m_{i=1}{n_i}+1=n+1$.
By the Lemma 3.3 above, we can conclude that $w_1^{i_1} \cdots w_{n+1}^{i_{n+1}}$ is zero for every $i_1,\cdots,i_{n+1}$ satisfying  $i_1+2i_2+\cdots+(n+1)i_{n+1}=n+1$, which implies that every Stiefel-Whitney number is zero.
\end{proof}
Unlike the real Bott manifolds, a small cover $N$ over $P^n\times \Delta^1$ may not bound equvariantly. Z.~L\"u and L.~Yu ~\cite{LuZhi11} give examples of small covers over $\Delta_2\times\Delta_1$, which do not bound equivariantly.

Indeed, there is a natural $\R P^1$-bundle on $N$. Note that $\R P^1\cong S^1$, if $N$ is a  $S^1$-principal bundle, we can directly deduce that $N$ is the boundary of a disk bundle. But actually we know the fact that if this $S^1$-bundle $N$ is not an orientable bundle, then this $S^1$-bundle $N$ is not principal. 
Furthermore, if $N$ is not orientable, while the related small cover over $P^n$ is orientable, this kind of $N$ is not a principal bundle. The following proposition tells us how to judge the orientability of a real Bott manifold in the terms of its \emph{reduce matrix}.
 \begin{prop} (Y.~Kamishima and M.~Masuda \cite [Lemma 2.2]{Kamishima09})
      The real Bott manifold $M^n$ is orientable if and only if the sum of entries is zero in $\Z_2$ for each column of $E_n+A$, where $A$ is the reduced matrix of $M^n$ and $E_n$ is the identity matrix of order $n$.
  \end{prop}
\begin{proof}  By the computations in Section 2,  we have the first Stiefel-Whitney class\\$w_1(M^n)=\sum_{j=1}^n y_j$. Let $k_{ij}$ be the $i$-th row and the $j$-th column element in the matrix $A+E_n$, then by the Equation (2) in Section 2,
\begin{align*}
 w_1(M^n)=\sum_{j=1}^n y_j=\sum_{j=1}^n \sum_{i=1}^n u_i k_{ij}=\sum_{i=1}^n u_i(\sum_{j=1}^n k_{ij}).
\end{align*}
Since a manifold is orientable if and only if its first Stiefel-Whitney class is zero, we could see that $M^n$ is orientable if and only if $\sum_{j=1}^n k_{ij}=0$ for every $i$.  
\end{proof}
It is easy to see that there exist many upper triangular matrices of order $n+1$ that the sum of entries is zero for the first $n$ columns, then we can just add the $n+1$-column with the sum not equal to $0$ and get the whole matrix $A+E_n$, where the $S^1$-bundle $N$ is not an orientable bundle. There is an example of such matrix,
     \begin{align*}
A=\begin{pmatrix}
        1& 0 &1  \\
            0 &  1 & 0\\
       0 & 0& 1\\
   \end{pmatrix}. 
     \end{align*} 
In this example, the related small cover over $P^n$ is a $2$-torus $T^2$, but $N$ is an unoriented $S^1$-bundle over $T^2$.
  

We can also give a sufficient condition to determine when the small cover over $P^n\times \Delta^l$ has identically zero Stiefel-Whitney numbers. This condition is only related to the elementary symmetric polynomials of $y_1,\cdots ,y_l$, where $y_j$ is defined by $v_{m+1}^{(j)}=u_{m+1}+y_j, 1\le j\le l$.
 \begin{thm}
Let $l=2^k-1$ be a fixed number. For any product of simplices $P^n$,  there exists a set of equations of $y_1,\cdots, y_l$ that the small cover over $P^n\times \Delta^l$ with the reduced matrix satisfying the equations will have identically zero Stiefel-Whitney numbers. 
 \end{thm} 

\begin{proof} Assume that $y_1,\cdots, y_l$ satisfy a set of equations $\{ \sigma_i=0|1\le i\le l, i\notin l+1-2^j, j=0\cdots k\}$, where  $\sigma_1,\cdots, \sigma_l$ are elementary symmetric polynomials of  $y_1,\cdots ,y_l$.\\
The Stiefel-Whitney class of $N$, \\
$ w(N)=(\prod_{i=1}^{m}(1+v_i^{(0)})\cdots(1+v_i^{(n_i)}))(1+u_{m+1})(1+u_{m+1}+y_1)\cdots(1+u_{m+1}+y_l).$
\begin{align*}
       (1+u_{m+1})(1+u_{m+1}+y_1)\cdots(1+u_{m+1}+y_l)
       &\\=(1+u_{m+1})^{l+1}+(1+u_{m+1})^l\sigma_1+\cdots+(1+u_{m+1})\sigma_l.
 \end{align*}
Only if $j$ is the power of $2$, $(1+u_{m+1})^j$ will exactly contain $1$ and the term of the highest degree. So, since $\sigma_i=0$, when $1\le i\le l, i\notin l+1-2^j$ and $j=0,\cdots ,k$, then 
 $(1+u_{m+1})(1+u_{m+1}+y_1)\cdots(1+u_{m+1}+y_l)$ will only have polynomials of elements in $u_1,\cdots,u_m$ when modulo the ideals $I$ and $J$. By almost the same process as the proof of the Theorem 3.4, we can conclude that every Stiefel-Whitney number of $N$ is zero.
\end{proof}

  \begin{exam}
When $k=2, l=3$, and $P^n$ is any product of simplices, the small cover over $P^n\times \Delta^3$ has identically zero Stiefel-Whitney numbers, if the corresponding \emph{reduced matrix} satisfies $\sigma_1=0$, i.e. $y_1+y_2+y_3=0$.\\
But generally, a small cover over $P^n\times \Delta^3$ may not have identically zero Stiefel-Whitney numbers. Take $\Delta^3\times \Delta^3$ as an example. If $y_1=y_2=y_3=u_1$, then 
the total Stiefel-Whitney class $w=(1+u_1)^4(1+u_1+u_2)^3(1+u_2)$. We can compute and get: $w_3^2=u_2^4(u_1+u_2)^2$. Since the cohomology ring of this manifold is $\Z_2[u_1,u_2]/\langle u_1^4, u_2(u_1+u_2)^3\rangle$, and $u_2^4(u_1+u_2)^2$ doesn't belong to the ideal $\langle u_1^4, u_2(u_1+u_2)^3\rangle$, then $w_3^2=u_2^4(u_1+u_2)^2$ is not zero.
  \end{exam}

    \section{Cobordism of generalized Bott manifold and quasitoric manifold}
  Suppose that $M^{2n}$ is a quasitoric manifold over the product of simplices $P^n=\Delta^{n_1}\times\cdots\times\Delta^{n_m}$ of dimension $n$. Then let the second integral cohomology classes $v_1，\cdots,v_{n+m}$ dual to the facial submanifolds of facets $F_1,\cdots,F_{n+m}$. For convenience, we will directly use the terms with the coefficient $\Z_2$ defined for the generalized real Bott manifolds, if there is no ambiguity. Furthermore, we have exactly the same equations as in Lemma 3.3. In this section, the \emph{reduced matrix} is an integer matrix,  and other terms are defined with $\Z$-coefficient if necessary. 

Then we have the following lemma, which is an easy corollary of Theorem 1.1.
 \begin{lem}
  If we project every element in the reduced matrix $A$ of $M^{2n}$ into $\Z_2$, we can get a $\Z_2$-coefficient matrix $\widetilde{A}$. If this matrix determines a small cover with identically zero Stiefel-Whitney numbers, then the reduced matrix $A$ determines a quasitoric manifold with identically zero Stiefel-Whitney numbers.
  \end{lem}
To extend our theorems above, we need to introduce the concept of generalized Bott manifold, which means it is not only a quasitoric manifold, but also a complex manifold. There are equivalent conditions for a quasitoric manifold to be a generalized real Bott manifold. 
 \begin{thm} (S.~Choi, M.~Masuda and D.Y.~Suh \cite [Theorem 6.4]{Suh09})
Let $M$ be a quasitoric manifold over a product of simplices $P^n$, and let $A$ be the reduced vector matrix associated with $M$, which has $\mathbf{1}$ as the diagonal entries. Then the following are equivalent:
\begin{itemize}
\item[(1)] $M$ is equivalent to a generalized Bott manifold.
\item[(2)] $M$ is equivalent to a quasitoric manifold which admits an invariant almost complex structure under the action.
\item[(3)]  the principal minors of $A$ are $1$.
\end{itemize}
 \end{thm}
The quasitoric version of Lemma 3.1(see S.~Choi, M.~Masuda and D.Y.~Suh ~\cite [Lemma 5.1]{Suh09}) confirms that the  reduced vector matrix of a generalized Bott manifold is conjugate to a unipotent triangular upper vector matrix. 
Then the Theorem 3.4 can be extended to the case of a quasitoric manifold.
 \begin{thm}
Let $P^n$ be a product of simplices. The generalized Bott manifold $M^{2(n+1)}$ over $P^n\times \Delta^1$ has identically zero Stiefel-Whitney numbers and Pontryagin numbers, i.e. is a  boundary.
 \end{thm}
\begin{proof} Owing to the Theorem 3.4 and the Lemma 4.1, the generalized Bott manifold $M^{2(n+1)}$ has identically zero Stiefel-Whitney numbers. Assume that $2(n+1)=4k$, otherwise the Pontryagin numbers will be zero. From the terms defined in the Section 3, the total Pontryagin class of $M^{2(n+1)}$ is 
$p(M^{2(n+1)})= (\prod_{i=1}^{m}(1-(v_i^{(0)})^2)\cdots(1-(v_i^{(n_i)})^2))(1-(v_{m+1}^{(0)})^2)(1-(v_{m+1}^{(1)})^2)$. 
By the restriction of the Stanley-Reinser ideal, $v_{m+1}^{(0)}v_{m+1}^{(1)}=0$, then $(1-(v_{m+1}^{(0)})^2)(1-(v_{m+1}^{(1)})^2)=1-(v_{m+1}^{(0)})^2-(v_{m+1}^{(1)})^2=1-(v_{m+1}^{(0)}+v_{m+1}^{(1)})^2$.
Since $v_{m+1}^{(1)}$ is the sum of $-v_{m+1}^{(0)}$ and a linear combination of $u_1,\cdots, u_m$,  $v_{m+1}^{(0)}+v_{m+1}^{(1)}$ is a linear combination of $u_1,\cdots, u_m$. Furthermore, every square term of $p(M^{2n})$, such as $(v_1^{(0)})^2,\cdots,(v_m^{(n_m)})^2, (v_{m+1}^{(0)}+v_{m+1}^{(1)})^2$, has degree $4$. Then every term $p_1^{i_1} \cdots p_k^{i_{k}}$ of degree $2(n+1)$ consists of terms which are products of square terms. Each of these products consists of homogenous polynomials of $u_1,\cdots, u_m$ with degree $2k=n+1$. So by the  Equation (3) in Lemma 3.3, every $p_1^{i_1} \cdots p_k^{i_{k}}$ of degree $2(n+1)$ is zero.
\end{proof}
If the principal minors of the \emph{reduced matrix} are not identically $1$, the quasitoric manifold $M^{2(n+1)}$ over a cube may or may not be null-cobordant. But, for the $4k$-dimensional quasitoric manifold $M$ over an even dimensional $n$-cube, if all the proper principal minors of the \emph{reduced matrix} are $1$, while the determinant of the whole \emph{reduced matrix} is $-1$, then this manifold is never a boundary. To state this conclusion, we need a lemma.
\begin{lem} (M.~Masuda and T.E.~Panov \cite [Theorem 3.3]{Masuda91})
Let $A$ be an $n \times n$ matrix with entries in ${\Z}$. Suppose that every proper principal minor of A is $1$. If $det A=-1$, then A is conjugated by a permutation matrix to the following matrix
     \begin{align*}
 \begin{pmatrix}
        1 & b_1&0 &\cdots& 0   \\
       0 & 1& b_2 &\cdots &0  \\
        \vdots & \vdots &\ddots & \ddots& \vdots \\
      0 &0 & \cdots & 1 &b_{n-1}\\
        b_n & 0 & \cdots & 0  & 1\\
   \end{pmatrix} 
     \end{align*}
where $b_i\neq 0.$
  \end{lem}
Since the dimension $n(=2k)$ of the cube is even and the determinant is $-1$, we have $b_1\cdots b_n=2$. Just as the case of real Bott manifold, Let $v_1,\cdots, v_n$ be the $2$-cohomology classes  dual to 
     the facial submanifolds of facets $F_1,\cdots,F_n$, meeting at a vertex.
Then we have facets $F_1^*,\cdots,F_n^*$ which are parallel to $F_1,\cdots,F_n$ respectively. Let the $2$-cohomology classes $u_1,\cdots,u_n$ be dual to the facial submanifolds of facets $F_1^*,\cdots,F_n^*$.
 \begin{thm}
For the $4k$-dimensional quasitoric manifold $M$ over $n$-cube, where $n=2k$ and $k$ is any positive number, if all the proper principal minors of the reduced matrix of $M$ are $1$, but the determinant of the reduced matrix of $M$ is $-1$, then this manifold has nonzero pontryagin numbers, i.e. is not a boundary.
 \end{thm}
\begin{proof}  By the Lemma 4.4, the \emph{reduced matrix} of $M$ is conjugated by a permutation matrix to the matrix in Lemma 4.4. By the restrictions given by $J$, there exist equations that: $-v_1=u_1+b_n u_n$ and $-v_i=u_i+b_{i-1} u_{i-1}$, for $2\le i \le n$. The total Pontryagin class of $M$ is 
\begin{align*}
p(M)&=(1-u_1^2)(1-v_1^2)\cdots(1-u_n^2)(1-v_n^2) =(1-(u_1+v_1)^2)\cdots(1-(u_n+v_n)^2)
\\ &=(1-(b_1 u_1)^2)\cdots(1-(b_n u_n)^2),      
\end{align*}
given by the restrictions of Stanley-Reinser Ideal $I$. This ideal also gives a group of relations: $u_1(u_1+b_n u_n)=0$ and $u_i(u_i+b_{i-1} u_{i-1})=0$, which determine the cohomology ring of $M$. The equality $b_1\cdots b_n=2$ implies that there exists exactly only one $b_l$ equal to $\pm2$, and other $b_j$ are $\pm1$. Then we consider $u_j^n$, where $j=1,2\cdots,n$. By the group of relations above, every $u_j^n$ will be equal to $u_1\cdots u_n$ multipled by a coefficient. For example, $u_l^n=\pm u_1\cdots u_n, u_{l-1}^n=\pm 2u_1\cdots u_n, \dots , u_{l+1}^n=\pm2^{n-1} u_1\cdots u_n$.  Then $\sum _{j=1}^n (b_j u_j)^n$ can be written as the sum of $\pm u_1\cdots u_n, \pm 2u_1\cdots u_n,\cdots,\pm2^n u_1\cdots u_n$. But by the equation $1+2+\cdots+2^{n-1}<2^n$, $\sum _{j=1}^n (b_j u_j)^n$ is a term of $u_1\cdots u_n$ multipled by an nonzero coefficient.\\
For elementary symmetric polynomials $\sigma_1,\cdots, \sigma_k$ and $s_k=\sum_{i=1}^n t_i^k$, the Newton-Girard formula states that:
 \begin{align*}
(-1)^k s_k /k=\sum_{i_1+2i_2+\cdots+ki_k=k}(-1)^{i_1+\cdots+i_k} \frac{(i_1+\cdots+i_k-1)!}{i_1!\cdots i_k!} \sigma_1^{i_1}\cdots \sigma_k^{i_k}
 \end{align*}
This formula can be used in the ring $H^*(M;\Z)\otimes \R$. Obviously, $p_1,\cdots,p_k$ are elementary symmetric polynomials of $(b_1 u_1)^2,\cdots,(b_n u_n)^2$, maybe with a sign. If all the Pontryagin numbers of $M$ are zero, then $\sum _{j=1}^n (b_j u_j)^{2n}$ is zero by the formula, which is a contradiction.
\end{proof}
The manifolds discussed in Theorem 4.5 are interesting. If it is over an odd dimensional cube, it is cobordant to zero by Lemma 4.1. But if it is over an even dimensional cube, it is unorientedly cobordant to zero, but is not orientedly cobordant to zero.

However, for the $4k$-dimensional quasitoric manifold $M$ over $2k$-cube, even if the proper principal minors of the \emph{reduced matrix} of $M$ are not all $1$, $M$ might be orientedly cobordant to zero.
  \begin{exam}
Assume that the \emph{reduced matrix} of $M$ is the following: 
     \begin{align*}
 \begin{pmatrix}
                       A_1& 0\\
                         0&A_2\\
   \end{pmatrix} , where\  A_1 \ and\ A_2 \ are \ of \ the \ form \ in \ Lemma\ 4.4.
     \end{align*}
Let the orders of $A_1$ and $A_2$, $n_1$ and $n_2$, be odd. The total Pontryagin class of $M$ is $p(M)=(1-(b_1 u_1)^2)\cdots(1-(b_{n_1} u_{n_1})^2)(1-(c_1 u_{n_1+1})^2)\cdots(1-(c_{n_2} u_{n_1+n_2})^2)$. Since every polynomial of degree more than $n_1$ over $u_1,\cdots, u_{n_1}$ is zero, and every polynomial of degree more than $n_2$ over $u_{n_1+1},\cdots, u_{n_1+n_2}$ is zero, we can deduce that the cohomology class $p_1^{i_1}\cdots p_k^{i_k}$ of degree $4k$ is zero.   
  \end{exam}

We can apply the method in Theorem 4.5 to give a sufficient condition for a generalized Bott manifold to have nonzero Pontryagin numbers. In the following theorem, we will use the terms defined in the discussion of the generalized real Bott manifold. To be clear, we claim that $y_i$ is defined by $v_m^{(i)}=-(u_m+y_i)$.
 \begin{thm}
For the $4k$-dimensional generalized Bott manifold $M$ over a product of simplices $P^n$, where $n=2k$ and $k$ is any positive number, if we have the equation $u_m^n+(u_m+y_1)^n+\cdots+(u_m+y_{n_m})^n\neq0$ in the integral cohomology ring of $M$, then this manifold has nonzero pontryagin numbers.
 \end{thm}
\begin{proof} The total Pontryagin class of $M$,
     \begin{align*}
p(M)=(\prod_{i=1}^{m-1}(1+(v_i^{(0)})^2)\cdots(1+(v_i^{(n_i)})^2))(1+(v_{m}^{(0)})^2)\cdots(1+(v_{m}^{(n_m)})^2)\\ =(\prod_{i=1}^{m-1}(1+(v_i^{(0)})^2)\cdots(1+(v_i^{(n_i)})^2))(1+u_m^2)\cdots(1+(u_m+y_{n_m})^2). 
     \end{align*}
By the method of the Theorem 4.5, if $\sum_{i=1}^{m-1} ((v_i^{(0)})^n+\cdots+(v_i^{(n_i)})^n)+u_m^n+(u_m+y_1)^n+\cdots+(u_m+y_{n_m})^n\neq0$, then there exist nonzero Pontryagin numbers. By the Lemma 3.3, every term in $\sum_{i=1}^{m-1} ((v_i^{(0)})^n+\cdots+(v_i^{(n_i)})^n)$ is zero, which ends the proof.
\end{proof} 

\textbf{Acknowledgments}   This work was motivated by a talk with Prof. L.~Yu, the author's master mentor in Nanjing University. The author would also like to thank him for his useful suggestions.

\end{document}